\documentclass[12pt]{article}

\usepackage{comment}
\usepackage{pgf,tikz}\usetikzlibrary{arrows}
\usepackage{tikzscale}
\usepackage{pgfplots}
\usepackage{cite}
\pgfplotsset{compat=1.8}
\usepackage{fancyhdr}
\usepackage{float}
\usepackage{framed}
\usepackage{amsthm}
\usepackage{amsmath}
\usepackage{amsfonts,dsfont}
\usepackage{amssymb}
\usepackage{color}
\usepackage[colorlinks=true, linkcolor=blue, urlcolor=blue, citecolor=black]{hyperref}
\usepackage[all,cmtip]{xy}
\usepackage[small,compact]{titlesec}

\usepackage{geometry}
\newcommand{\md}{\mathrm{d}}

\newcommand{\J}{{\rm J}}
\newcommand{\loc}{\text{loc}}

\newcommand{\pp}[2]{\frac{\partial #1}{\partial #2}}

\theoremstyle{definition}\newtheorem{definition}{Definition}
\theoremstyle{definition}\newtheorem{proposition}[definition]{Proposition}
\theoremstyle{definition}\newtheorem{remark}[definition]{Remark}
\theoremstyle{definition}\newtheorem{theorem}[definition]{Theorem}
\theoremstyle{definition}
\theoremstyle{definition}
\theoremstyle{definition}\newtheorem{example}[definition]{Example}
\theoremstyle{definition}
\theoremstyle{definition}
\theoremstyle{definition}

\makeatletter
\renewcommand*\env@matrix[1][c]{\hskip -\arraycolsep
  \let\@ifnextchar\new@ifnextchar
  \array{*\c@MaxMatrixCols #1}}
\makeatother		

\makeatletter
\renewcommand*\env@matrix[1][*\c@MaxMatrixCols c]{%
  \hskip -\arraycolsep
  \let\@ifnextchar\new@ifnextchar
  \array{#1}}
\makeatother

\allowdisplaybreaks[4]

\begin{document}

\thispagestyle{fancy}
\fancyhead{}
\fancyfoot{}
\renewcommand{\headrulewidth}{0pt}
\cfoot{\thepage}
\rfoot{\today}

\begin{center}
\begin{Large}
Symmetry-preserving finite element schemes:\\  An introductory investigation
\end{Large}

\vskip 0.5cm

\begin{tabular*}{0.9\textwidth}{@{\extracolsep{\fill}} ll}
Alexander Bihlo
& Francis Valiquette\\
Department of Mathematics and Statistics & Department of Mathematics\\
Memorial University of Newfoundland& SUNY at New Paltz \\
St. John's, NL, Canada\quad A1C 5S7& New Paltz, NY, USA \quad 12561\\
{\tt abihlo@mun.ca} & {\tt valiquef@newpaltz.edu} \\
{\tt http://www.math.mun.ca/$\sim$abihlo} &{\tt http://www2.newpaltz.edu/$\sim$valiquef}
\end{tabular*}
\end{center}

\hspace{0.4cm}\parbox{0.9\textwidth}{
}

\vskip 0.5cm

\abstract{}
Using the method of equivariant moving frames, we present a procedure for constructing symmetry-preserving finite element methods for second-order ordinary differential equations. Using the method of lines, we then indicate how our constructions can be extended to (1+1)-dimensional evolutionary partial differential equations, using Burgers' equation as an example.
Numerical simulations verify that the symmetry-preserving finite element schemes constructed converge at the expected rate and that these schemes can yield better results than their non-invariant finite element counterparts.

\section{Introduction}

Geometric numerical integration is a branch of numerical analysis dedicated to the construction of numerical schemes that preserve intrinsic geometric properties of the differential equations being approximated, \cite{HLW-2006}.  Standard examples include symplectic integrators, \cite{BC-2016,HLW-2006,LR-2004,SC-1994}, Lie--Poisson structure preserving schemes, \cite{ZM-1988}, energy-preserving methods, \cite{QM-2008}, and general conservative methods, \cite{WBN-2016,WBN-2017}. The motivation for considering structure-preserving numerical schemes is that, as a rule of thumb, these integrators provide better global or long term results than their traditional non-geometric counterparts.

In engineering, physics, mathematics, and other mathematical sciences, most  differential equations of interest admit a group of symmetries that encapsulates properties of the equations and their solution spaces.  Over the last 30 years, there has been a considerable amount of work dedicated to the development of finite difference numerical methods that preserve the Lie point symmetries of differential equations, \cite{BDK-1997,BD-2001,D-1989,DW-2000,HLW-2000}. For ordinary differential equations, symmetry-preserving numerical schemes have shown to be very effective, especially when solutions exhibit sharp variations or admit singularities, \cite{BCW-2006,BRW-2008,CRW-2016,KO-2004}. For partial differential equations, the numerical improvements are not as clear and more work remains to be done, \cite{BCW-2015,K-2008,LMW-2015,RV-2013}. For evolutionary partial differential equations, symmetry-preserving schemes generally require the use of time-evolving meshes which can lead to mesh tangling and other numerical instabilities. To avoid these mesh singularities, various methods have been proposed in recent years, including evolution--projection techniques, invariant $r$-adaptive methods, and invariant meshless discretizations, \cite{B-2013,BN-2013,BN-2014,BP-2012}.

To this day, research on symmetry-preserving numerical schemes has solely focused on finite difference methods. Extending the methodology of symmetry-preserving schemes to other numerical integration techniques such as finite volumes, finite elements, or spectral methods remains to be done.  As such, in this paper we lay out basic ideas for constructing symmetry-preserving finite element methods.  

From a numerical perspective, finite element methods offer several advantages over finite difference methods.  For example, when dealing with complex domains, unstructured grids, or moving boundaries, finite element methods are generally easier to implement than finite difference methods. Also,  the finite element method relies on discretizing a weak form of a system of differential equations and thus has less rigid smoothness requirements than methodologies that rely on the discretization of the strong form of the system. As a result, finite element methods are extensively employed in computational fluid dynamics, structural mechanics, and many other branches of engineering, physics, and applied mathematics.  

As a first attempt to systematically construct symmetry-preserving finite element methods, we restrict our considerations to second-order ordinary differential equations.  As basis functions, we consider piecewise linear (Lagrangian) functions (also called hat functions).  For more accurate schemes, our constructions can easily be applied to higher order Lagrangian interpolants. The ideas developed in this paper can also be applied to higher order ordinary differential equations, with appropriate interpolating functions.  Adapting our results to hierarchical bases, splines and Hermite basis functions, and partial differential equations with multi-dimensional basis functions remains to be considered.

The remainder of this paper is laid out as follows.  In Section \ref{problem section} we state the problem we aim to solve in this paper.  Namely, we show via two examples that, in general, the discrete weak formulation of a differential equation will not preserve the symmetries of the original differential equation.  To remedy this situation, we explain how to construct symmetry-preserving finite element schemes using the method of equivariant moving frames.  The basic moving frame constructions, adapted to the problem at hand, are introduced in Section \ref{moving frame section}.  The main results of this paper are found in Section~\ref{symmetry-preserving FEM section}, where we provide an algorithm for constructing symmetry-preserving finite element schemes.   In Section \ref{ODE section}, our constructions are illustrated with several examples of ordinary differential equations. Numerical results are presented that verify the convergence of the proposed invariant finite element schemes and show that symmetry-preserving finite element schemes can provide better numerical results than their non-invariant counterparts.  In Section \ref{PDE section} we explain how to adapt the constructions introduced for ordinary differential equations to (1+1)-dimensional evolutionary partial differential equations using the method of lines. This is illustrated using Burgers' equation as an example. Finally, in Section \ref{section conclusion} we summarize our findings and give some directions for future research.

\section{Statement of the Problem}\label{problem section}

Let $x \in \mathbb{R}$ be the independent variable, and $u=u(x)$ a real-valued scalar function. In the following we consider single second-order ordinary differential equations written in the form
\begin{equation}\label{ODE}
u_{xx} = \Delta(x,u,u_x).
\end{equation}
Here and in what follows, we are using the index notation for derivatives, i.e.\ 
\[
u_x=\frac{\md u}{\md x}\qquad \text{and}\qquad u_{xx}=\frac{\md^2 u}{\md x^2}.
\] 
Now, let $G$ be an $r$-parameter Lie group acting locally on the plane $\mathbb{R}^2$ parametrized by $(x,u)$.  Using capital letters to denote the transformed variables, we have 
\begin{equation}\label{G action}
X = g\cdot x\qquad\text{and}\qquad U = g\cdot u,\quad \text{where}\quad g\in G.
\end{equation}
The group action \eqref{G action} induces a \emph{prolonged action} on the derivatives given by the chain rule:
\[
U_X = \frac{\mathrm{D}_x(U)}{\mathrm{D}_x(X)},\qquad U_{XX} = \frac{\mathrm{D}_x(U_X)}{\mathrm{D}_x (X)},
\]
where
\[
\mathrm{D}_x = \pp{}{x} + u_x\pp{}{u} + u_{xx} \pp{}{u_x} + \cdots
\]
is the total derivative operator with respect to the independent variable $x$.

\begin{definition}
A local Lie group of transformations $G$ acting on an open subset of $\mathbb{R}^2$ is said to be a \emph{symmetry group} of the differential equation \eqref{ODE} if the solution space of the equation in invariant under the given group action.  In other words,
\begin{equation}\label{symmetry definition}
U_{XX} = \Delta(X,U,U_X)\qquad \text{whenever}\qquad u_{xx} = \Delta(x,u,u_x).
\end{equation}
\end{definition}

The main goal of this paper consists of recasting \eqref{ODE} into its weak form, and to introduce a systematic procedure for constructing a discrete approximation of the weak form that will preserve the symmetries of the original differential equation. To achieve this goal, we introduce the space of real-valued locally integrable functions on $\mathbb{R}$,
\[
L_{1,\loc}(\mathbb{R}) = \Big\{f\colon \mathbb{R}\to \mathbb{R}\;\Big|\; \int_K |f|\, \mathrm{d}x < \infty \text{ for all compact subsets } K \subset \mathbb{R}\Big\}.
\]
Alternatively, $L_{1,\loc}(\mathbb{R})$ is defined as the set of functions $f\colon\mathbb{R}\to \mathbb{R}$ such that for any compactly supported test function $\phi \in C_c^\infty(\mathbb{R})$, the integral
\[
\int_{-\infty}^\infty |f\phi|\, \mathrm{d}x < \infty
\]
is finite.  We now assume that solutions to \eqref{ODE} and their derivatives are in $L_{1,\loc}(\mathbb{R})$.  Multiplying the differential equation \eqref{ODE} by a test function $\phi \in C_c^\infty(\mathbb{R})$, and integrating over $\mathbb{R}$, we obtain, using integration by parts, the weak formulation of equation \eqref{ODE}:
\begin{equation}\label{ODE weak form}
0 = \int_\infty^\infty [-u_{xx} + \Delta(x,u,u_x)]\phi\, \mathrm{d}x = \int_{-\infty}^\infty [u_x\phi_x + \Delta(x,u,u_x)\phi]\, \mathrm{d}x.
\end{equation}

Let $G$ be the symmetry group of the differential equation \eqref{ODE}.  The group $G$ acts on the test function $\phi$ via the usual group action on functions:
\[
\Phi = g\cdot \phi = \phi(g\cdot x) = \phi(X),\qquad g\in G.
\]
The induced action on $\phi_x$ is given by the chain rule
\[
\Phi_X = \frac{\phi_x}{\mathrm{D}_x X}.
\]
The group $G$ also acts on the differential $\mathrm{d}x$, \cite{FO-1999}.  The action is given by
\begin{equation}\label{dx transformation}
\omega = g\cdot \mathrm{d}x =  (\mathrm{D}_xX)\, \mathrm{d}x.
\end{equation}

Restricting our attention to local Lie group actions, \cite{O-1993}, we assume that $g\in G$ is near the identity element so that the bounds of integration in \eqref{ODE weak form} remain infinite once an element of the symmetry group acts on the weak form.  The following theorem is essential for our consideration of invariant finite element discretizations.

\begin{theorem}
The invariance of the differential equation \eqref{ODE} implies the invariance of the weak form \eqref{ODE weak form}. 
\end{theorem}

\begin{proof}
Indeed, 
\begin{align*}
\int_{-\infty}^\infty [U_X\Phi_X + \Delta(X,U,U_X)\Phi]\, \omega 
 = \int_{-\infty}^\infty [-U_{XX} + \Delta(X,U,U_X)]\Phi\, \omega 
 = \int_{-\infty}^\infty 0\cdot \Phi\, \omega = 0,
\end{align*}
where \eqref{symmetry definition} was used.
\end{proof}

This theorem is essential as it guarantees that for a given differential equation with symmetry group $G$, the Lie group $G$ remains a symmetry group of its corresponding weak form.  Therefore, when seeking to construct a symmetry-preserving numerical scheme for a particular differential equation, one can either start with the original strong form or work with a suitable weak form. The strong form of a differential equation is the starting point for constructing symmetry-preserving finite difference schemes, which is the route that has been taken so far in the literature, \cite{BDK-1997,B-2013,BCW-2015,BN-2013,BN-2014,BP-2012,BV-2017,BCW-2006,BRW-2008,BD-2001,CRW-2016,D-1989,DKW-2000,DW-2000,HLW-2000,K-2008,KO-2004,LMW-2015,O-2001,RV-2013}. On the other hand, the weak form is the starting point for constructing symmetry-preserving finite element schemes, which is the focus of the present paper.

To approximate \eqref{ODE weak form}, we subdivide the real line $\mathbb{R}$ into the \emph{elements} $[x_n, x_{n+1}]$. For an introduction to the theory of finite elements, we refer the reader to \cite{BS-2007}.  In this paper, the space of test functions $C_c^\infty(\mathbb{R})$ is replaced by the space of hat functions
\[
H^\md = \{ \phi_k\colon\mathbb{R} \to \mathbb{R}\;|\; k\in \mathbb{Z}\},
\]
where\\[0.25cm]
\begin{minipage}{0.45\linewidth}
\begin{center}
\begin{tikzpicture}[scale = 0.75]
\begin{axis}[
 axis lines=middle,
 xmin=-0.1,xmax=2.5,
 xtick={0.25,1.25,2.25},xticklabels={$x_{k-1}$,$x_k$,$x_{k+1}$},
 ymin=-0.1,ymax=1.1,
 ytick={1},yticklabels={1},
 samples=200]
\addplot[domain=0.25:1.25,thick] {x-0.25};
\addplot[domain=1.25:2.25,thick] {2.25-x};
\end{axis}
\node at (3.75,5.5) {$\phi_k(x)$};
\end{tikzpicture}
\end{center}
\end{minipage}\hfill
\begin{minipage}{0.55\linewidth}
\begin{equation}\label{hat functions}
\phi_k(x) = \begin{cases}
\cfrac{x-x_{k-1}}{x_k - x_{k-1}} & x \in [x_{k-1},x_k] \\
\cfrac{x_{k+1} - x}{x_{k+1} - x_k} & x \in [x_k,x_{k+1}] \\
\hskip 0.75cm 0 & x \notin [x_{k-1},x_{k+1}]
\end{cases}.
\end{equation}
\end{minipage}
The solution $u(x)$ to the weak formulation \eqref{ODE weak form} is now approximated by the (infinite) linear combination 
\begin{equation}\label{ud}
u(x) \approx u^\md(x) = \sum_{k=-\infty}^\infty u_k \phi_k(x),
\end{equation}
where $u_k = u(x_k)$ denotes the value of the function $u(x)$ at the node $x_k$.  A first order approximation of the weak form \eqref{ODE weak form} is then given by
\begin{equation}\label{discrete weak form}
0 = \int_{-\infty}^\infty [u_x^\md \phi^\prime_k + \Delta(x,u^\md,u^\md_x)\phi_k]\,\mathrm{d}x,
\end{equation}
where $\phi_k^\prime = \mathrm{D}_x(\phi_k)$ denotes the derivative of $\phi_k$, and 
\[
u_x^\md(x) = \sum_{k=-\infty}^\infty u_k \phi_k^\prime(x)
\]
approximates the first derivative $u_x$.

The transformation group \eqref{G action} induces an action on the discrete weak form \eqref{discrete weak form}.  The action on the nodes $x_k$ and the coefficients $u_k$ in the expansion \eqref{ud} is given by the product action
\[
X_k = g\cdot x_k,\qquad U_k = g\cdot u_k,\qquad k\in \mathbb{Z},
\]
and the action on the hat function $\phi_k(x)$ is
\[
\Phi_k = g\cdot\phi_k = \phi_k(X) = \begin{cases}
\cfrac{X-X_{k-1}}{X_k - X_{k-1}} & X \in [X_{k-1},X_k] \\
\cfrac{X_{k+1} - X}{X_{k+1} - X_k} & X \in [X_k,X_{k+1}] \\
\hskip 0.85cm 0 & X \notin [X_{k-1},X_{k+1}]
\end{cases}.
\]
We then introduce the transformed interpolated function
\begin{equation}\label{Ud}
U^\md(X):= \sum_{k=-\infty}^\infty U_k\Phi_k(X).
\end{equation}
In the subsequent developments, we require the transformed approximation \eqref{Ud} to be of the same form as the original interpolant \eqref{ud}.  In other words, we require $U^\md$ to be a linear combination of basis functions that depend solely on the independent variable $x$.  This can be achieved by requiring that $\Phi_k(X)$ is a function of $x$ only and not of $u$.  To do so, we require the group action to be \emph{projectable}, \cite{O-1993}. This assumption requires the transformation rule in the independent variable to be a function of $x$ (and the group parameters):
\[
X = g\cdot x = X(x,g),\qquad U=g\cdot u = U(x,u,g).
\]

The reason for requiring the group action to be projectable comes from the fact that if this were not the case, then the transformed hat function $\Phi_k$ would depend on the unknown function $u(x)$, which would make it impossible to evaluate the integral in \eqref{discrete weak form}, and therefore make it impossible to obtain the corresponding finite element scheme.  Investigating the possibility of extending the constructions to general, non-projectable group actions remains to be done.  We do stress here though that most symmetry groups of differential equations arising as models in the mathematical sciences are indeed projectable, and thus the projectability assumption captures essentially all equations of practical relevance.

With this in mind, we now recall a theorem due to Lie, \cite{HA-1975,L-1880}.

\begin{theorem}
The largest Lie group contained in the diffeomorphism pseudo-group of the real line $\mathcal{D}(\mathbb{R})$ is the special linear group $SL(2,\mathbb{R})$.  Up to a local diffeomorphism, the action of $SL(2,\mathbb{R})$ on the real line $\mathbb{R}$ is given by fractional linear transformations:
\begin{equation}\label{SL2 action}
X = \frac{\alpha x + \beta}{\gamma x + \delta},\qquad \alpha\delta - \beta \gamma = 1.
\end{equation}
\end{theorem}

\begin{proposition}
Under the fractional linear transformation \eqref{SL2 action} the hat function $\phi_k(x)$ transforms according to the formula
\begin{subequations}\label{basis functions transformation}
\begin{equation}\label{Phi_k}
\Phi_k(X) = \phi_k(x) \cdot \frac{\gamma x_k + \delta}{\gamma x + \delta},
\end{equation}
while the transformation rule for the first derivative is 
\begin{equation}\label{Phi_k-prime}
\Phi_k^\prime(X) = (\gamma x_k+\delta)[(\gamma x+\delta) \phi_k^\prime(x) - \gamma \phi_k(x)],
\end{equation}
where $\phi_k(x)$ is differentiable.
\end{subequations}
\end{proposition}

\begin{proof}
Formula \eqref{Phi_k} is obtained by substituting \eqref{SL2 action} into the definition of the hat function $\phi_k(x)$ in \eqref{hat functions}.  As for \eqref{Phi_k-prime}, the chain rule yields
\begin{align*}
\Phi_k^\prime(X) &= \frac{1}{\mathrm{D}_x X} \mathrm{D}_x[\Phi_k(X)] = (\gamma x+ \delta)^2\mathrm{D}_x\bigg[\phi_k(x)\cdot \frac{\gamma x_k + \delta}{\gamma x + \delta}\bigg]\\
&= (\gamma x_k+\delta)[(\gamma x+\delta) \phi_k^\prime(x) - \gamma \phi_k(x)].
\end{align*}
\end{proof}

Under the action \eqref{SL2 action}, formula \eqref{dx transformation}  for the induced action on the differential $\mathrm{d}x$ becomes
\begin{equation}\label{SL2 dx transformation}
\omega = (\mathrm{D}_xX)\, \mathrm{d}x = \frac{\mathrm{d}x}{(\gamma x + \delta)^2}.
\end{equation}

Knowing how each constituent of the discrete weak form \eqref{discrete weak form} transforms under the action of the Lie group $G$, we can now address the main purpose of the paper.  Given a second-order ordinary differential equation with symmetry group $G$ and weak form~\eqref{ODE weak form},  we seek to construct, in a systematic fashion, a weak form approximation that will remain invariant under the symmetry group of the differential equation.  In general, the naive discretization \eqref{discrete weak form} will not preserve all the symmetries of the continuous problem.  To construct a symmetry-preserving discrete weak form we will use the method of equivariant moving frames, \cite{FO-1999,M-2010,MM-2016}, which is endowed with an \emph{invariantization map} that can be used to map non-invariant quantities to their invariant counterparts.  In our case, we will use the invariantization map to invariantize the discrete weak form~\eqref{discrete weak form}, resulting in a symmetry-preserving finite element scheme.

\begin{example}\label{non-invariance linear ODE example}
As a simple example, consider the second-order linear ordinary differential equation
\begin{equation}\label{linear ODE}
u_{xx} + p(x)u_x+q(x)u=f(x),
\end{equation}
where $p$, $q$, and $f$ are arbitrary smooth functions of their argument. Equation \eqref{linear ODE} admits a two-parameter symmetry group given by
\begin{equation}\label{linear symmetry}
X=x,\qquad U = u+\epsilon_1\alpha(x) + \epsilon_2\gamma(x),
\end{equation}
where $\alpha(x)$ and $\gamma(x)$ are two linearly independent solutions of the homogeneous equation $u_{xx} + p(x)u_x+q(x)u=0$.  The corresponding weak form of~\eqref{linear ODE} is
%
\[
\int_{-\infty}^\infty [-u_x\phi_x + (p(x)u_x+q(x)u-f(x))\phi]\, \mathrm{d}x = 0,
\]
%
while an approximation to this weak form is given by
\begin{equation}\label{linear ODE discrete weak form}
\int_{-\infty}^\infty [-u_x^\md \phi_k^\prime + (p(x) u_x^\md + q(x)u^\md - f(x))\phi_k]\, \mathrm{d}x = 0.
\end{equation}
Acting on the latter with the group action \eqref{linear symmetry}, we obtain
\begin{equation}\label{non-invariance}
\begin{aligned}
0 &= \int_{-\infty}^\infty [-u_x^\md \phi_k^\prime + (p(x) u_x^\md + q(x)u^\md - f(x))\phi_k]\, \mathrm{d}x \\
&\hskip 1cm + \epsilon_1 \int_{-\infty}^\infty [-\alpha_x^\md \phi_k^\prime + (p(x) \alpha_x^\md + q(x)\alpha^\md)\phi_k]\,\mathrm{d}x \\
&\hskip 1cm + \epsilon_2 \int_{-\infty}^\infty [-\gamma_x^\md \phi_k^\prime + (p(x) \gamma_x^\md + q(x)\gamma^\md)\phi_k]\, \mathrm{d}x,
\end{aligned}
\end{equation}
where
\[
\alpha^\md = \sum_{k=-\infty}^\infty \alpha_k \phi_k(x),\qquad \alpha^\md_x = \sum_{k=-\infty}^\infty \alpha_k \phi_k^\prime(x),\qquad \alpha_k = \alpha(x_k),
\]
and similarly for $\gamma^\md$ and $\gamma^\md_x$.  Since the last two integrals in \eqref{non-invariance} are, in general, nonzero, the discrete weak form \eqref{linear ODE discrete weak form} does not admit the superposition principle given by \eqref{linear symmetry}.
\end{example}

\begin{example}
As a less trivial example, consider the  second-order nonlinear ordinary differential equation
\begin{equation}\label{nonlinear ODE}
u_{xx} = \frac{1}{u^3}.
\end{equation}
This equation is invariant under the group action
\begin{equation}\label{symmetry nonlinear ODE}
X = \frac{\alpha x + \beta}{\gamma x + \delta},\qquad U = \frac{u}{\gamma x + \delta},\qquad \alpha \delta - \beta \gamma = 1,
\end{equation}
and a weak formulation of \eqref{nonlinear ODE} is given by
\begin{equation}\label{nonlinear ode weak form}
0 = \int_{-\infty}^\infty \bigg[u_x \phi_x + \frac{1}{u^3}\phi\bigg]\, \mathrm{d}x.
\end{equation}
Approximating $u^{-3}$ by
\[
\frac{1}{u^3} \approx \sum_{k=-\infty}^\infty \frac{1}{u_k^3}\phi_k(x),
\]
we obtain the discrete weak form
\begin{equation}\label{discrete nonlinear ODE weak form}
0 = \int_{-\infty}^\infty \bigg[\sum_{\ell = -\infty}^\infty \bigg(u_\ell \phi_\ell^\prime \phi_k^\prime + \frac{1}{u_\ell^3}\phi_\ell \phi_k\bigg)\bigg]\, \mathrm{d}x.
\end{equation}
Acting on \eqref{discrete nonlinear ODE weak form} with the symmetry group \eqref{symmetry nonlinear ODE}, recalling \eqref{basis functions transformation} and \eqref{SL2 dx transformation}, we obtain, after simplification,
\begin{equation}\label{non-invariance discrete weak form}
0=\int_{-\infty}^\infty \bigg[\sum_{\ell = -\infty}^\infty\bigg( u_\ell \phi_\ell^\prime \phi_k^\prime + \frac{1}{u_\ell^3}\bigg(\frac{\gamma x_\ell + \delta}{\gamma x + \delta}\bigg)^4 \phi_\ell \phi_k\bigg)\bigg]\, \mathrm{d}x.
\end{equation}
The extra factor $\bigg(\cfrac{\gamma x_\ell + \delta}{\gamma x + \delta}\bigg)^4$ in the second term of \eqref{non-invariance discrete weak form} shows that the discrete weak form \eqref{discrete nonlinear ODE weak form} is not invariant under the group action \eqref{symmetry nonlinear ODE}.
\end{example}

We conclude this section by observing that all our considerations can be restricted to boundary value problems, which are more standard in the application of the finite element method.  Instead of working on the whole real line $\mathbb{R}$, simply restrict all considerations to an interval $[a,b]$ and impose boundary conditions at $x=a$ and $x=b$.  The symmetry group $G$ should now consist of all (or a subset of all) transformations that keep the differential equation and its boundary conditions invariant.  As the boundary conditions impose further constraints, the symmetry group of the boundary value problem will usually be smaller than the symmetry group of the differential equation itself, \cite{BA-2002}.  A slightly less restrictive assumption is to allow symmetry transformations of a given system of differential equations without boundary conditions to act as equivalence transformations preserving a class of boundary value problems containing the problem under consideration, \cite{BP-2012}.

\section{Moving Frames}\label{moving frame section}

The theoretical foundations of the discrete equivariant moving frame method have recently been developed in \cite{MM-2016,O-2001}. For the sake of completeness of the present exposition, we summarize the theory of moving frames relevant to the construction of symmetry-preserving finite element schemes here. 

After evaluating the discrete weak form \eqref{discrete weak form}, the result is a function of the discrete points $(x_{k-1},u_{k-1})$, $(x_k,u_k)$, and $(x_{k+1},u_{k+1})$.  In the following, we combine these three points into the \emph{second-order discrete jet} at $k \in \mathbb{Z}$:
\[
z^{[2]}_k = (k,x_{k-1},u_{k-1},x_k,u_k,x_{k+1},u_{k+1}).
\]
The terminology stems from the fact that $z^{[2]}_k$ contains sufficiently many points to approximate the function $u(x)$ and its derivatives $u_x$, $u_{xx}$ at the node $x_k$ using central differences.  We introduce the \emph{second-order discrete jet space} 
\[
\J^{[2]} = \bigcup_{k=-\infty}^\infty\, z_k^{[2]}, 
\]
which consists of the union of the second-order discrete jets over the integers $k \in \mathbb{Z}$.  The discrete jet space $\J^{[2]}$ admits the structure of a \emph{lattice variety} or \emph{lattifold}, which is a manifold-like object modeled on $\mathbb{Z}$ rather than $\mathbb{R}$, \cite{MM-2016}.  Alternatively, $\J^{[2]}$ is a disconnected manifold with fibers isomorphic to the Euclidean space $\mathbb{R}^6$.  In the following, we let $\pi\colon \J^{[2]} \to \mathbb{Z}$ denote the projection onto the discrete index $k$:
\[
\pi(z^{[2]}_k) = k.
\]

Now, let $G$ be an $r$-parameter Lie group acting on the plane $\mathbb{R}^2=\{(x,u)\}$.  Extending the action trivially to $\mathbb{Z}$,
\[
g\cdot k = k,
\]
the Lie group $G$ induces an action on the discrete jet $z^{[2]}_k$ via the product action
\begin{equation}\label{product action}
\begin{aligned}
Z^{[2]}_k &= g\cdot z^{[2]}_k \\
(k,X_{k-1},U_{k-1},X_k,U_k,X_{k+1},U_{k+1})  &= 
(k,g\cdot x_{k-1},g\cdot u_{k-1}, g\cdot x_k, g\cdot u_k, g\cdot x_{k+1}, g\cdot u_{k+1}).
\end{aligned}
\end{equation}
See \cite{BV-2017} for further details. In other words, the Lie group $G$ induces an action on each fiber of $\J^{[2]}$ via the product action.  In the following, we assume that the action is \emph{(locally) free} and \emph{regular} on each fiber $\pi^{-1}(k)=\J^{[2]}|_k$.  This forces $\dim G \leqslant \dim \J^{[2]}|_k = 6$.  We recall that the product action is free at $z_k^{[2]}$ if the isotropy group
\[
G_{z_k^{[2]}} = \{ g \in G\;|\; g\cdot z_k^{[2]} = z_k^{[2]}\} = \{e\}
\]
is trivial, and that the action is locally free at $z_k^{[2]}$ if the isotropy group is discrete.  On the other hand, the action is regular if the group orbits have the same dimension and each point in $\J^{[2]}|_k$ has arbitrarily small neighborhoods whose intersection with each orbit is a connected subset thereof.

\begin{definition}
Let $G$ act (locally) freely and regularly on (each fiber of) $\J^{[2]}$ by the product action \eqref{product action}.  A \emph{discrete (right) moving frame} is a $G$-equivariant map $\rho\colon \J^{[2]} \to G$ satisfying 
\[
\rho(g\cdot z_k^{[2]}) = \rho(z_k^{[2]})\, g^{-1}
\]
for all $g \in G$ where the product action is defined.
\end{definition}

The construction of a discrete moving frame is based on the introduction of a (collection of) cross-section(s) $\mathcal{K} \subset \J^{[2]}$ to the group orbits.

\begin{definition}
A subset $\mathcal{K} \subset \J^{[2]}$ is a \emph{cross-section} to the group orbits if for each $k \in \mathbb{Z}$, the restriction $\mathcal{K}|_k \subset \J^{[2]} = \pi^{-1}(k)$ is a submanifold of $\J^{[2]}|_k$, transverse and of complementary dimension to the group orbits.
\end{definition}

In general, a cross-section $\mathcal{K} \subset J^{[2]}$ is specified by a system of $r=\dim G$ difference equations
%
\[
\mathcal{K} = \{ E_\ell(z_k^{[2]}) = 0\; |\; \ell = 1,\ldots,r\}.
\]
%
The right moving frame $\rho(z_k^{[2]})$ at $z_k^{[2]}$ is then the unique group element in $G$ that sends $z_k^{[2]}$ onto $\mathcal{K}|_k$:
\[
\rho(z_k^{[2]})\cdot z_k^{[2]} \in \mathcal{K}|_k.
\]
The coordinate expressions of the moving frame are obtained by solving the \emph{normalization equations}
\[
E_\ell(g\cdot z_k^{[2]}) = 0,\qquad \ell = 1,\ldots,r,
\]
for the group parameters $g=(g_1,\ldots,g_r)$.

Given a moving frame, there is a systematic procedure for constructing invariant functions, invariant differential forms, and other invariant quantities, \cite{FO-1999}.  

\begin{definition}
Let $\rho\colon\J^{[2]} \to G$ be a right moving frame.  The \emph{invariantization} of the difference function $F(k,x_i,u_i,\ldots,x_j,u_j)$ is the invariant
\begin{equation}\label{invariantization}
\iota_k(F)(k,x_i,u_i,\ldots,x_j,u_j) = F(k,\rho_k\cdot x_i,\rho_k\cdot u_i,\ldots,\rho_k\cdot x_j,\rho_k\cdot u_j)
\end{equation}
obtained by acting on the arguments of $F$ with the moving frame $\rho_k=\rho(z_k^{[2]})$. 
\end{definition}

Borrowing the notation from \cite{M-2010}, we can rewrite \eqref{invariantization} as
\[
\iota_k(F)(k,x_i,u_i,\ldots,x_j,u_j) = F(k,g\cdot x_i,g\cdot u_i,\ldots g\cdot x_j, g\cdot u_j)\big|_{g=\rho_k}.
\]
Thus, the invariantization of $F$ is obtained by first acting on its argument by an arbitrary group element $g \in G$, followed by the substitution $g=\rho_k$.  In particular, the invariantization of the components of a point $(x_\ell,u_\ell)$ are the invariants
\[
\iota_k(x_\ell) = g\cdot x_\ell\big|_{g=\rho_k},\qquad \iota_k(u_\ell) = g\cdot u_\ell\big|_{g=\rho_k}.
\]
Similarly, we can also invariantize the hat function $\phi_\ell(x)$ and its derivative:
\[
\iota_k(\phi_\ell) = (g\cdot \phi_\ell)\big|_{g=\rho_k},\qquad 
\iota_k(\phi_\ell^\prime) = (g\cdot \phi_\ell^\prime)\big|_{g=\rho_k}.
\]
Therefore, the invariantization of $u^\md(x)$ and $u^\md_x$ is
\[
\iota_k(u^\md) = \sum_{\ell = -\infty}^\infty \iota_k(u_\ell)\iota_k(\phi_\ell),\qquad
\iota_k(u^\md_x) = \sum_{\ell = -\infty}^\infty \iota_k(u_\ell)\iota_k(\phi^\prime_\ell).
\]
Finally, according to \eqref{dx transformation}, the invariantization of the one-form $\mathrm{d}x$ is the invariant one-form
\[
\iota_k(\mathrm{d}x) = (\mathrm{D}_x X)\big|_{g=\rho_k}\, \mathrm{d}x.
\]
The one-form $\iota_k(\mathrm{d}x)$ is invariant since we limit our considerations to projectable group actions.  For general group actions, $\iota_k(\mathrm{d}x)$ would be contact-invariant\footnote{A differential form $\Omega$ on the jet space $\J^{(n)}$ is said to be \emph{contact-invariant} if and only if, for every $g \in G$, $g^*\Omega = \Omega + \theta_g$ for some contact form $\theta_g$, \cite{O-1995}.}.

\begin{example}\label{moving frame example}
As an example of the moving frame construction introduced above, let us consider the group action \eqref{symmetry nonlinear ODE}.  Introducing the centered difference derivative
\[
u_x^k = \frac{u_{k+1}-u_{k-1}}{x_{k+1}-x_{k-1}},
\]
a cross-section on $\J^{[2]}$ is given by 
\[
\mathcal{K} = \{x_k = 0,\; u_k = 1, u_x^k = 0 \}.
\] 
Solving the normalization equations
\[
0=X_k = \frac{\alpha x_k + \beta}{\gamma x_k + \delta},\qquad
1=U_k = \frac{u_k}{\gamma x_k + \delta},\qquad
0=U_X^k = \gamma(\overline{x}_k u_x^k - \overline{u}_k)+\delta u_x^k,
\]
for the group parameters, where 
\[
\overline{x}_k = \frac{x_{k+1}+x_{k-1}}{2},\qquad
\overline{u}_k = \frac{u_{k+1}+u_{k-1}}{2},
\]
we obtain the discrete moving frame
\begin{equation}\label{moving frame ex}
\alpha = \frac{1}{u_k},\qquad 
\beta = -\frac{x_k}{u_k},\qquad 
\gamma = \frac{u_k u_x^k}{(x_k-\overline{x}_k)u_x^k + \overline{u}_k},\qquad
\delta = \frac{u_k[\overline{u}_k - \overline{x}_k u_x^k]}{(x_k - \overline{x}_k)u_x^k + \overline{u}_k}.
\end{equation}
Invariantizing $u_\ell$, we obtain the invariant
\begin{equation}\label{u invariantization}
\iota_k(u_\ell) = \frac{u_\ell}{\gamma x_\ell + \delta}\bigg|_{\eqref{moving frame ex}} = \frac{u_\ell[(x_k - \overline{x}_k)u_x^k + \overline{u}_k]}{u_k[(x_\ell - \overline{x}_k)u_x^k + \overline{u}_k]},
\end{equation}
while the invariantization of the hat function $\phi_\ell(x)$ is
\begin{equation}\label{phi invariantization}
\iota_k(\phi_\ell(x)) = \phi_\ell \cdot \frac{\gamma x_\ell + \delta}{\gamma x +\delta}\bigg|_{\eqref{moving frame ex}} = \phi_\ell \cdot \frac{(x_\ell-\overline{x}_k)u_x^k + \overline{u}_k}{(x-\overline{x}_k)u_x^k + \overline{u}_k}.
\end{equation}
Combining \eqref{u invariantization} and \eqref{phi invariantization}, we obtain the invariantization of $u^\md$:
\[
\iota_k(u^\md) =  \sum_{\ell=-\infty}^\infty \frac{u_\ell}{u_k}\cdot \frac{(x_k-\overline{x}_k)u_x^k + \overline{u}_k}{(x-\overline{x}_k)u_x^k + \overline{u}_k}\, \phi_\ell(x).
\]
Finally, the invariantization of the one-form $\mathrm{d}x$ is
\[
\iota_k(\mathrm{d}x) = \frac{\mathrm{d}x}{(\gamma x + \delta)^2}\bigg|_{\eqref{moving frame ex}} = \frac{[(x_k - \overline{x}_k)u_x^k + \overline{u}_k]^2}{u_k^2[(x-\overline{x}_k)u_x^k + \overline{u}_k]^2}\, \mathrm{d}x.
\]
\end{example}

\section{Symmetry-Preserving Finite Element Schemes}\label{symmetry-preserving FEM section}

Given a second-order ordinary differential equation of the form \eqref{ODE} with projectable symmetry group~$G$, we now have everything in hand to construct a symmetry-preserving finite element scheme.  First, rewrite the differential equation in its weak form \eqref{ODE weak form}.  Then, consider the discrete approximation \eqref{discrete weak form} or any other suitable approximation.  In general, the discrete weak form will not preserve all the symmetries of the differential equation.  To obtain a symmetry-preserving finite element scheme, first construct a discrete moving frame for the symmetry group $G$ as explained in Section \ref{moving frame section}.  Then use the corresponding invariantization map to invariantize the discrete weak form \eqref{discrete weak form} or any suitable approximation.

To guarantee the consistency of the symmetry-preserving finite element scheme, we need to impose certain constraints on the general moving frame constructions introduced in Section \ref{moving frame section}.  Namely, in the continuous limit where the lengths of the elements $[x_{k-1},x_k]$ and $[x_k,x_{k+1}]$ go to zero, all discrete constructions need to converge to their continuous counterparts.  To guarantee this convergence, we have to construct a consistent moving frame compatible with a differential moving frame, \cite{O-2001}.  In other words, the discrete moving frame should, in the continuous limit, converge to a moving frames defined for the prolonged action of $G$ on the submanifold jet $\J^{(2)}=\{(x,u,u_x,u_{xx})\}$, \cite{O-1993}.  This will be the case if the cross-section $\mathcal{K}$ used to define the discrete moving frame converges, in the continuous limit, to a cross-section in $\J^{(2)}$.  In practice, this can be accomplished by using solely, for example, the approximations
\begin{equation}\label{derivative approximations}
x_k,\qquad u_k,\qquad u_x^k = \frac{u_{k+1}-u_{k-1}}{x_{k+1}-x_{k-1}},\qquad
u_{xx}^k = \frac{2}{x_{k+1}-x_{k-1}}(u_x^k - u_x^{k-1})
\end{equation}
to define a discrete cross-section as in the continuous limit those quantities converge to $x$, $u$, $u_x$, and $u_{xx}$, respectively.

\subsection{Ordinary Differential Equations}\label{ODE section}

In this section we consider several ordinary differential equations to illustrate the construction of symmetry-preserving finite element discretizations. 

We note that all the schemes presented below are implicit and hence require the solution of a (nonlinear) algebraic equation. For this purpose, we used Newton's method with a termination tolerance of $10^{-15}$ in all numerical examples.

\subsubsection{Equation $u_{xx}=\exp(-u_x)$}

As our first example, we consider the equation
\begin{equation}\label{exponential ODE}
u_{xx} = \exp(-u_x).
\end{equation}
This equation admits the three-parameter symmetry group action
\begin{equation}\label{symmetry group exponential ODE}
X= e^\epsilon x + a,\qquad U = e^\epsilon u + \epsilon e^\epsilon x + b,\qquad \epsilon, a, b \in \mathbb{R}.
\end{equation}
A weak form formulation of equation \eqref{exponential ODE} is given by
\begin{equation}\label{weak form exponential ODE}
0 = \int_{-\infty}^\infty \big(u_x \phi_x + e^{-u_x}\phi)\, \mathrm{d}x.
\end{equation}
An approximation of \eqref{weak form exponential ODE} is provided by
\begin{equation}\label{discrete weak form exponential ODE}
0 = \int_{-\infty}^\infty \big(u^\md_x \phi_k^\prime + e^{-u^\md_x}\phi_k)\, \mathrm{d}x.
\end{equation}
We now show that the discrete weak form \eqref{discrete weak form exponential ODE} is already invariant under the group action \eqref{symmetry group exponential ODE}.  First, we have
\[
U_X^\md = u_x^\md + \epsilon,\qquad 
\Phi_k = \phi_k,\qquad 
\Phi_k^\prime = \frac{1}{e^\epsilon}\phi_k^\prime,\qquad
\omega = e^\epsilon\, \mathrm{d}x.
\]
Therefore
\begin{align*}
0 &= \int_{-\infty}^\infty [U_X^\md \Phi_k^\prime + e^{-U_X^\md}\Phi_k]\, \omega
= \int_{-\infty}^\infty \bigg[(u_x^\md + \epsilon)\frac{\phi_k^\prime}{e^\epsilon} + e^{-(u_x^\md + \epsilon)}\phi_k\bigg] e^\epsilon \mathrm{d}x\\
&= \int_{-\infty}^\infty [ u_x^\md \phi_k^\prime + e^{-u_x^\md}\phi_k]\, \mathrm{d}x + \epsilon\int_{-\infty}^\infty \phi_k^\prime\, \md x = \int_{-\infty}^\infty [ u_x^\md \phi_k^\prime + e^{-u_x^\md}\phi_k]\, \mathrm{d}x,
\end{align*}
since $\displaystyle \int_{-\infty}^\infty \phi_k^\prime\, \md x = 0$.  Evaluating the integral \eqref{discrete weak form exponential ODE} we obtain the symmetry-preserving finite element scheme
\begin{equation}\label{exponential ODE scheme}
(\Delta x_k+\Delta x_{k-1})\,u_{xx}^k = \Delta x_{k-1} \exp[-u_x^{k-1}] + \Delta x_k \exp[-u_x^k],
\end{equation}
where 
\[
\Delta x_k = x_{k+1}-x_k,\qquad 
u_x^k = \frac{u_{k+1}-u_k}{x_{k+1}-x_k},\qquad
u_{xx}^k = \frac{2}{x_{k+1}-x_{k-1}}[u_x^k - u_x^{k-1}].
\]
We observe that the finite element scheme \eqref{exponential ODE scheme} differs from the two schemes appearing in \cite{DKW-2000} (equations (4.25) and (4.26)).

We now test the invariant scheme \eqref{exponential ODE scheme} numerically,  by treating equation \eqref{exponential ODE} as an initial value problem. First, we note that the exact solution to the equation \eqref{exponential ODE} is
\[
u_{\rm a}(x)=(x+c_1)\ln(x+c_1)-x+c_2,
\]
where $c_1$ and $c_2$ are two arbitrary constants.  Using the initial conditions $u(0)=1$ and $u_x(0)=0$, the exact solution becomes $u_{\rm a}(x)=(x+1)\ln(x+1)-x+1$. Integrating \eqref{exponential ODE scheme}  from $x=0$ to $x=1$, the convergence plot of the relative $l_\infty$-error is shown in Figure \ref{figure convergence plot exponential}. As it can be seen, the scheme converges at first order, in accordance with the derivation of the finite element scheme,  which is based on a first order linear interpolant.

\begin{figure}[!ht]
\centering
\includegraphics[scale=0.6]{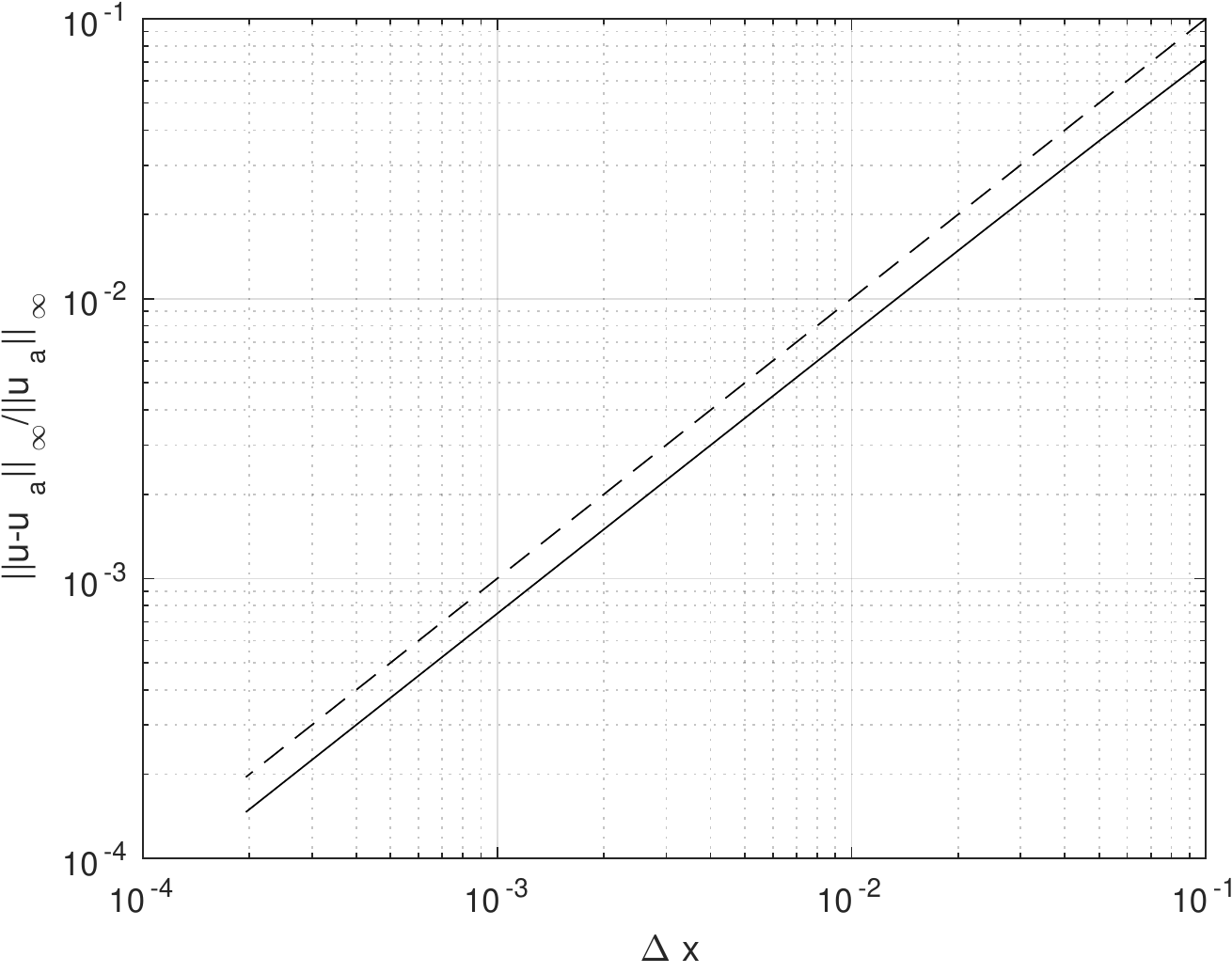}
\caption{\footnotesize{Convergence plot for the invariant numerical scheme~\eqref{exponential ODE scheme}.\textit{ Solid line:} relative $l_\infty$-error over the integration interval $[0,1]$ with initial conditions $u(0)=1$ and $u_x(0)=0$. \textit{Dashed line:} line of slope 1.}}
\label{figure convergence plot exponential}
\end{figure}

\subsubsection{Equation $u_{xx}+p(x)u_x+q(x)u=f(x)$}

In Example \ref{non-invariance linear ODE example}, we observed that the discrete weak formulation \eqref{linear ODE discrete weak form} does not preserve the linear superposition principle \eqref{linear symmetry} for the linear equation \eqref{linear ODE}.  To solve this problem, we now construct a symmetry-preserving finite element scheme.

The first step is to construct a moving frame.  As in \eqref{derivative approximations}, let
\[
\alpha_x^k = \frac{\alpha_{k+1}-\alpha_{k-1}}{x_{k+1}-x_{k-1}},\qquad 
\gamma_x^k = \frac{\gamma_{k+1} - \gamma_{k-1}}{x_{k+1}-x_{k-1}}
\]
denote the centered first derivative approximations.  In the following, we assume that
\[
\gamma_k = \gamma(x_k) \neq 0,\qquad \gamma_k\alpha_x^k - \alpha_k \gamma_x^k\neq 0.
\]
The second constraint is a discrete approximation of the Wronskian condition requiring that the solutions $\alpha(x)$ and $\gamma(x)$ are linearly independent.  We construct a moving frame by choosing the cross-section
\[
\mathcal{K} = \{u_k = u_x^k = 0\},
\]
where $u_x^k$ is the centered approximation introduced in \eqref{derivative approximations}.
Solving the normalization equations
\[
0 = U_k =  u_k + \epsilon_1\alpha_k + \epsilon_2\gamma_k,\qquad
0 = U^k_X = u_x^k + \epsilon_1 \alpha_x^k + \epsilon_2 \gamma^k_x,
\]
for the group parameters $\epsilon_1$ and $\epsilon_2$, we obtain
\begin{equation}\label{linear ODE moving frame}
\epsilon_1 = \frac{u_k \gamma_x^k - \gamma_k u_x^k}{\gamma_k \alpha_x^k - \alpha_k \gamma_x^k},\qquad
\epsilon_2 = \frac{\alpha_k u_x^k - u_k \alpha_x^k}{\gamma_k \alpha_x^k - \alpha_k \gamma_x^k}.
\end{equation}
Given the moving frame \eqref{linear ODE moving frame}, we invariantize the non-invariant discrete weak form \eqref{linear ODE discrete weak form}.  This is done by substituting the group normalizations \eqref{linear ODE moving frame} into \eqref{non-invariance}.  The result is the invariant discrete weak form
\begin{equation}\label{invariant linear ODE weak form}
\begin{aligned}
0 &= \int_{-\infty}^\infty [-u_x^\md \phi_k^\prime + (p(x) u_x^\md + q(x)u^\md - f(x))\phi_k]\, \mathrm{d}x \\
&\hskip 1cm + \frac{u_k \gamma_x^k - \gamma_k u_x^k}{\gamma_k \alpha_x^k - \alpha_k \gamma_x^k} \int_{-\infty}^\infty [-\alpha_x^\md \phi_k^\prime + (p(x) \alpha_x^\md + q(x)\alpha^\md)\phi_k]\,\mathrm{d}x \\
&\hskip 1cm + \frac{\alpha_k u_x^k - u_k \alpha_x^k}{\gamma_k \alpha_x^k - \alpha_k \gamma_x^k} \int_{-\infty}^\infty [-\gamma_x^\md \phi_k^\prime + (p(x) \gamma_x^\md + q(x)\gamma^\md)\phi_k]\, \mathrm{d}x.
\end{aligned}
\end{equation}

For second-order linear homogeneous equations, i.e.\ when $f(x)=0$ in \eqref{linear ODE}, we notice that
\begin{equation}\label{linear ODE y^h}
u^\md(x) = c_1\, \alpha^\md(x) + c_2\, \gamma^\md(x),
\end{equation}
where $c_1$ and $c_2$ are two arbitrary constants, is an exact solution of the discrete weak form \eqref{invariant linear ODE weak form}.  Indeed, when $u^\md(x)$ is given by \eqref{linear ODE y^h}, we have that
\[
\frac{u_k \gamma_x^k - \gamma_k u_x^k}{\gamma_k \alpha_x^k - \alpha_k \gamma_x^k} =  -c_1,\qquad \frac{\alpha_k u_x^k - u_k \alpha_x^k}{\gamma_k \alpha_x^k - \alpha_k \gamma_x^k} = -c_2,
\]
and the right-hand side of \eqref{invariant linear ODE weak form} is identically zero.

%
 %
%

\subsubsection{Equation $u_{xx}=u^{-3}$}

As a third example, we consider the nonlinear differential equation \eqref{nonlinear ODE}, with discrete weak form \eqref{discrete nonlinear ODE weak form}.  In Example \ref{moving frame example}, we computed a moving frame for the symmetry group \eqref{symmetry nonlinear ODE}.  The result is given in equation \eqref{moving frame ex}.  Invariantizing the discrete weak form \eqref{discrete nonlinear ODE weak form}, which is obtained by substituting the group parameter normalizations \eqref{moving frame ex} into the transformed discrete weak form \eqref{non-invariance discrete weak form}, we get the symmetry-preserving discrete weak form
\[
0 = \int_{-\infty}^\infty \bigg[\sum_{\ell=-\infty}^\infty \bigg(u_\ell \phi_\ell^\prime\phi_k^\prime + \frac{1}{u_\ell^3}\bigg(\frac{(x_\ell - \overline{x}_k)u_x^k + \overline{u}_k}{(x-\overline{x}_k)u_x^k + \overline{u}_k}\bigg)^4\phi_\ell \phi_k\bigg)\bigg]\, \mathrm{d}x.
\]
Integrating this expression yields the symmetry-preserving finite element scheme
\begin{multline}\label{cubic invariant finite element scheme}
-\bigg(\frac{u_{k+1}-u_k}{x_{k+1}-x_k}\bigg) + \bigg(\frac{u_k-u_{k-1}}{x_k - x_{k-1}}\bigg) + 
\frac{(x_k - x_{k-1})[(x_{k-1}-\overline{x}_k)u_x^k + \overline{u}_k]^2}{6u_{k-1}^3[(x_k-\overline{x}_k)u_x^k + \overline{u}_k]^2} \\
+\frac{(x_{k+1}-x_{k-1})[(x_k-\overline{x}_k)u_x^k + \overline{u}_k]^2}{3u_k^3[(x_{k+1}-\overline{x}_k)u_x^k + \overline{u}_k][(x_{k-1}-\overline{x}_k)u_x^k + \overline{u}_k]}
+\frac{(x_{k+1}-x_k)[(x_{k+1}-\overline{x}_k)u_x^k + \overline{u}_k]^2}{6u_{k+1}^3[(x_k-\overline{x}_k)u_x^k + \overline{u}_k]^2}=0.
\end{multline}

We now turn to the numerical verification of the resulting invariant scheme. First, we note that the general solution to the differential equation \eqref{nonlinear ODE} is
\[
u_{\rm a}^2(x)=\frac{1}{c_1}+c_1(x+c_2)^2,
\]
where $c_1$ and $c_2$ are arbitrary constants with $c_1\ne0$, \cite{P-1950}.

We integrate equation \eqref{nonlinear ODE} on the interval $[0,1]$ using the invariant finite element scheme \eqref{cubic invariant finite element scheme} and the initial conditions $u(0)=1$, $u_x(0)=0$. In this case the exact solution reduces to $u_{\rm a}(x)=\sqrt{1+x^2}$.   The convergence plot for the scheme~\eqref{cubic invariant finite element scheme} is presented in Figure~\ref{figure convergence plot}. As expected, this invariant scheme converges at first order, since it is based on a linear interpolant.

\begin{figure}[!ht]
\centering
\includegraphics[scale=0.6]{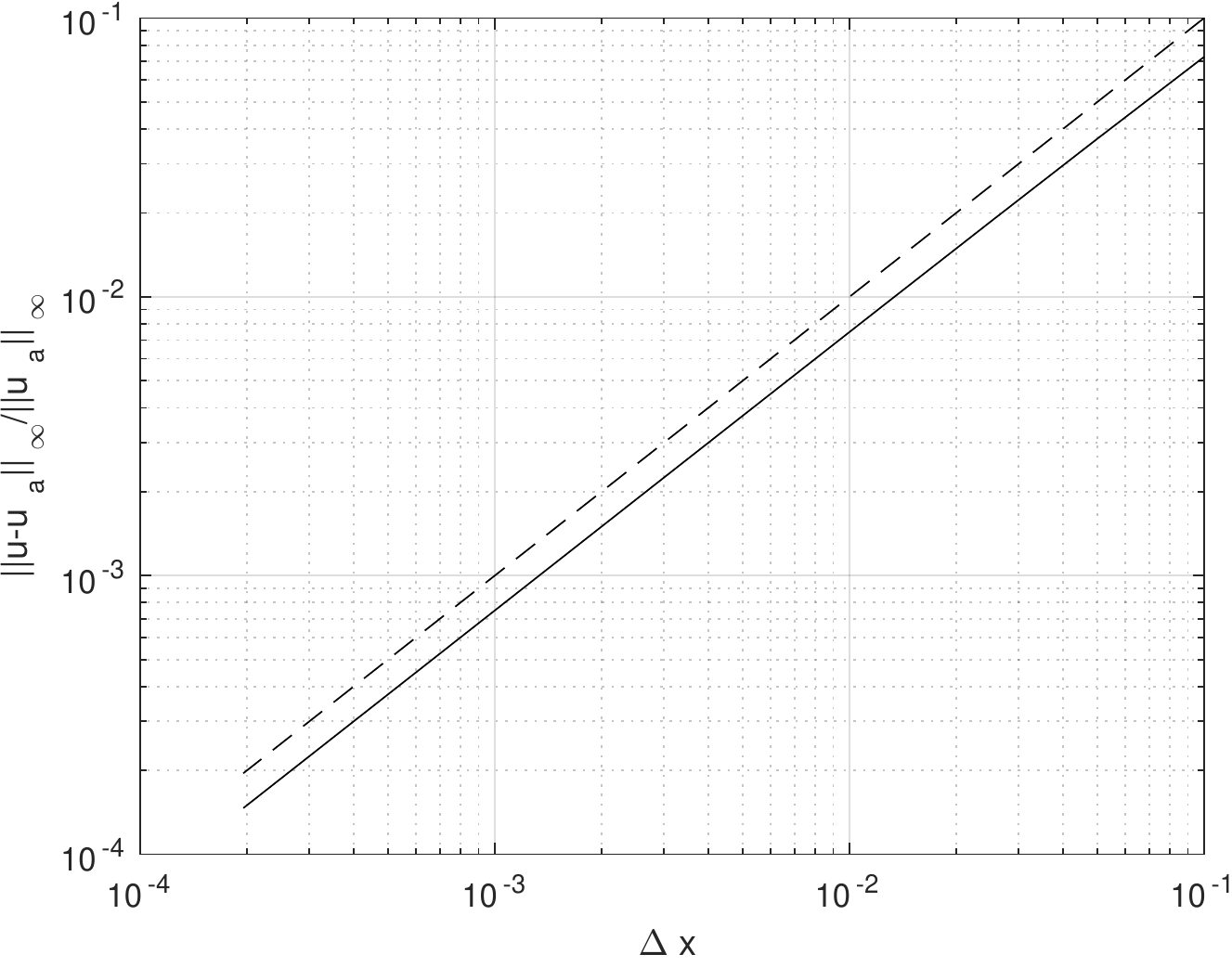}
\caption{\footnotesize{Convergence plot for the invariant numerical scheme~\eqref{cubic invariant finite element scheme}.\textit{ Solid line:} relative $l_\infty$-error over the integration interval $[0,1]$ with initial conditions $u(0)=1$ and $u_x(0)=0$. \textit{Dashed line:} line of slope 1.}}
\label{figure convergence plot}
\end{figure}

\begin{remark}
As an alternative discretization of the weak form \eqref{nonlinear ode weak form}, one could also consider the approximation
%
\[
\int_{-\infty}^\infty [u^\md_x\phi_k^\prime + (u^\md)^{-3}\phi_k]\, \mathrm{d}x = 0.
\]
%
Under the group action \eqref{symmetry nonlinear ODE}, this weak form gets dilated to
\[
0 = \int_{-\infty}^\infty[U_X^\md \Phi_k^\prime + (U^\md)^{-3}\Phi_k]\, \omega 
= (\gamma x_k + \delta) \int_{-\infty}^\infty [u_x^\md\phi_k^\prime + (u^\md)^{-3}\phi_k]\, \mathrm{d}x,
\]
and therefore preserves the symmetries of the original weak form \eqref{nonlinear ode weak form}.  The corresponding symmetry-preserving finite element scheme is
%
\[
\frac{u_{k+1}-u_k}{x_{k+1}-x_k} - \frac{u_k - u_{k-1}}{x_k - x_{k-1}} = \frac{x_{k+1}-x_k}{2u_k^2 u_{k+1}} + \frac{x_k-x_{k-1}}{2u_k^2 u_{k-1}}.
\]
\end{remark}

\subsubsection{Painlev\'e Equation $u_{xx} = u^{-1}u_x^2$}
As our final example we consider the Painlev\'e equation
\begin{equation}\label{painleve}
u_{xx} = \frac{u_x^2}{u}.
\end{equation}
This equation admits a six-parameter symmetry group of projectable transformations given by
\[
X = \frac{\alpha x + \beta}{\gamma x + \delta},\qquad
U = \big(u^{\lambda} e^{ax+b}\big)^{1/(\gamma x + \delta)},
\]
where $\alpha \delta - \beta \gamma = 1$, $a, b \in \mathbb{R}$, and $\lambda >0$.  In the following, we restrict our attention to the two-dimensional symmetry group
\begin{equation}\label{2d painleve symmetry group}
X=x,\qquad U = u e^{ax+b},
\end{equation}

A weak formulation of the Painlev\'e equation \eqref{painleve} is given by
\[
0 = \int_{-\infty}^\infty \bigg[ u_x \phi_x + \frac{u_x^2}{u}\phi\bigg] \md x.
\]
In the discrete setting, we approximate the weak form by
\begin{equation}\label{painleve discrete weak form}
0=\int_{-\infty}^\infty \bigg[ u_x^\md \phi_k^\prime + \frac{(u_x^\md)^2}{u^\md}\phi_k\bigg]\md x.
\end{equation}
Integrating \eqref{painleve discrete weak form}, we obtain the non-invariant finite element scheme
\begin{equation}\label{painleve non-invariant finite element scheme}
-2\bigg[\bigg(\frac{u_{k+1}-u_k}{x_{k+1}-x_k}\bigg) - \bigg(\frac{u_k - u_{k-1}}{x_k-x_{k-1}}\bigg)\bigg] + \frac{u_{k-1}}{x_k - x_{k-1}}\ln\bigg(\frac{u_{k-1}}{u_k}\bigg) + \frac{u_{k+1}}{x_{k+1}-x_k}\ln\bigg(\frac{u_{k+1}}{u_k}\bigg) = 0.
\end{equation}
To construct a symmetry-preserving finite element scheme, we construct a moving frame to the group action \eqref{2d painleve symmetry group} using the cross-section
\[
\mathcal{K} = \bigg\{ u_k = 1,\; u_x^k = \frac{u_{k+1}-u_{k-1}}{x_{k+1}-x_{k-1}} = 0\bigg\}.
\]
Solving the corresponding normalization equations, we obtain
\[
a = \frac{1}{x_{k+1} - x_{k-1}}\ln\bigg(\frac{u_{k-1}}{u_{k+1}}\bigg),\qquad
b = \frac{x_k}{x_{k+1}-x_{k-1}} \ln\bigg(\frac{u_{k+1}}{u_{k-1}}\bigg) - \ln u_k.
\]
Invariantizing the discrete weak form \eqref{painleve discrete weak form} and performing the integration we obtain the symmetry-preserving finite element scheme
\begin{subequations}\label{painleve invariant scheme}
\begin{equation}
-2\bigg[\bigg(\frac{I_k-1}{x_{k+1}-x_k}\bigg)-\bigg(\frac{1-J_k}{x_k-x_{k-1}}\bigg)\bigg]+ \frac{J_k\ln J_k}{x_k-x_{k-1}}+\frac{I_k\ln I_k}{x_{k+1}-x_k} = 0,
\end{equation}
where the invariants $I_k$ and $J_k$ are given by
\begin{equation}
I_k = \frac{u_{k+1}}{u_k}\exp\bigg[-\frac{x_{k+1}-x_k}{x_{k+1}-x_{k-1}} \ln\bigg(\frac{u_{k+1}}{u_{k-1}}\bigg)\bigg],\qquad
J_k = \frac{u_{k-1}}{u_k} \exp\bigg[\frac{x_k-x_{k-1}}{x_{k+1}-x_{k-1}}\ln\bigg(\frac{u_{k+1}}{u_{k-1}}\bigg)\bigg].
\end{equation}
\end{subequations}

We now compare the invariant scheme \eqref{painleve invariant scheme} against the non-invariant scheme \eqref{painleve non-invariant finite element scheme} numerically. Since the symmetry-preserving scheme is exact, i.e.\ the only difference between the numerical solution and the exact solution is due to round-off error, we do not need to verify the convergence of the scheme. A straightforward Taylor series analysis reveals that the non-invariant finite element scheme for the Painlev\'e equation is of second order.

We now solve the initial value problem for the Painlev\'e equation with initial conditions $u(0)=1$ and $u_x(0)=1$, corresponding to the exact solution $u=\exp(x)$.  Integrating over the interval $[0,1]$ using a step size of $\Delta x=0.01$, the time series of the relative error between the numerical solutions of the two schemes \eqref{painleve non-invariant finite element scheme}, \eqref{painleve invariant scheme} and the exact solution is depicted in Figure~\ref{figure error plot Painleve}.  It is obvious that the symmetry-preserving scheme outperforms the non-invariant scheme, with the error of the invariant scheme being several magnitudes smaller and approximately of the size of machine epsilon.

\begin{figure}[!ht]
\centering
\includegraphics[scale=0.6]{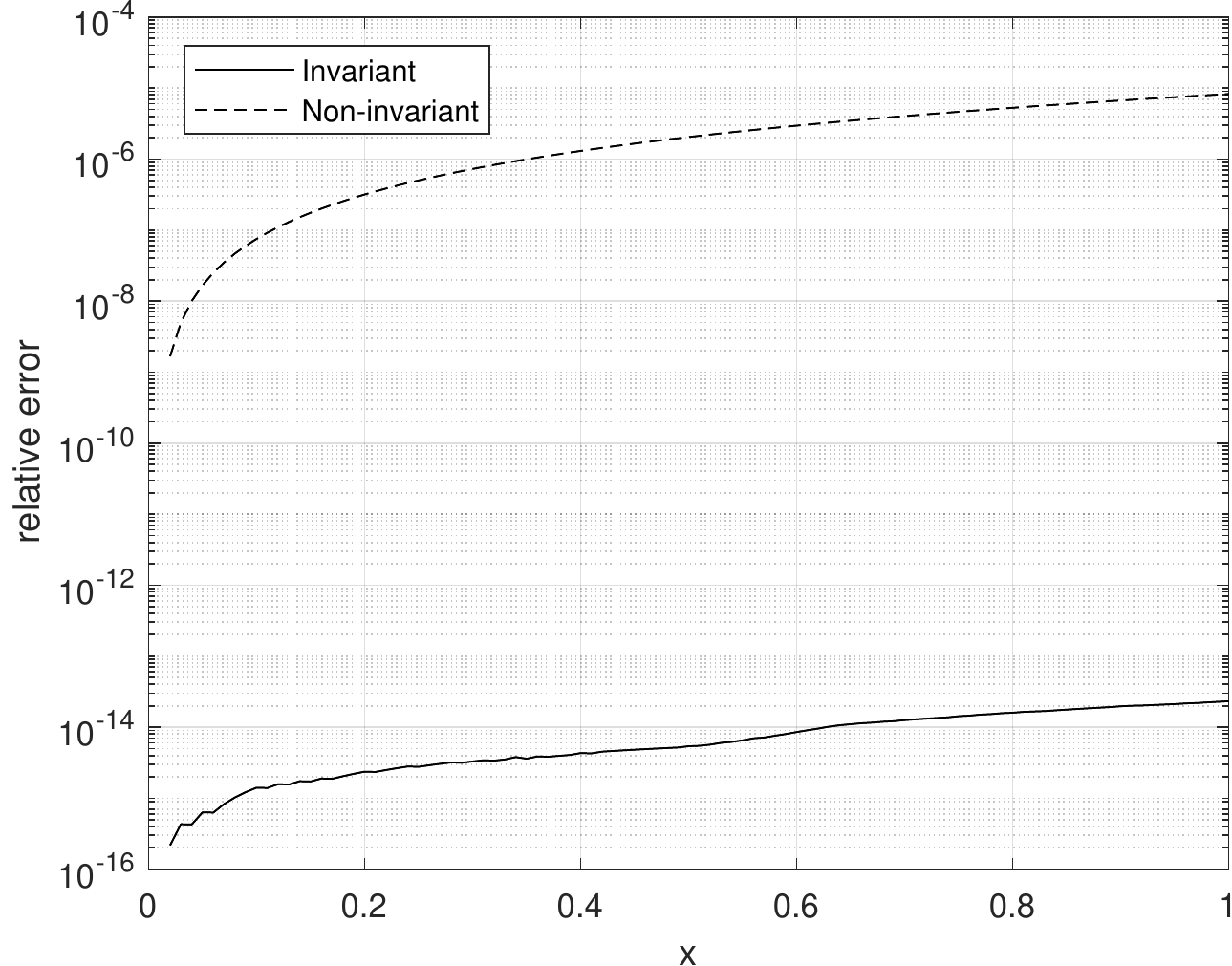}
\caption{\footnotesize{Time series of relative error for the invariant finite element scheme (solid line) and the non-invariant finite element scheme (dashed line). The initial conditions were $u(0)=1$ and $u_x(0)=1$, and we integrated the Painlev\'e equation up to $x=1$, interpreted as initial value problem, using a step size of $\Delta x=0.01$.}}
\label{figure error plot Painleve}
\end{figure}

\begin{remark}
We note that numerically solving the nonlinear algebraic equation \eqref{painleve invariant scheme} for the invariant finite element method is challenging due to the fact that this scheme is exact. Numerically, we observe an accumulation of round-off errors that is growing over the integration interval. The smaller the step size $\Delta x$, the more round-off error can accumulate. To numerically preserve the exactness of the scheme for all step sizes $\Delta x$, variable precision arithmetic may be necessary.
\end{remark}

\subsection{Partial Differential Equations}\label{PDE section}

In this section we extend the constructions introduced in the previous sections to the semi-discretization of (1+1)-dimensional evolution equations, where only the spatial variable is discretized.  This allows us to use many of the ideas introduced in the previous sections. To simplify the exposition, we focus on a particular example and consider Burgers' equation
\begin{equation}\label{Burgers}
u_t + u u_x = \nu u_{xx},\qquad \text{where}\qquad \nu >0,
\end{equation}
which plays an important role in various areas of applied mathematics, such as fluid mechanics, nonlinear acoustics, gas dynamics, and traffic flow. Here $\nu$ is the constant viscosity coefficient. Burgers' equation admits a five-parameter maximal Lie symmetry group, see e.g.~\cite{O-1993}.  One of these admitted symmetry transformations yields an inversion of time, which does not respect the requirement that the time variable $t$ should increase monotonically for a given initial value problem, \cite{BP-2012}. Thus, we restrict our attention to the four-parameter  subgroup of symmetry transformations
\begin{equation}\label{Burgers group action}
X= \lambda (x + v t) + a,\qquad T = \lambda^2 t + b,\qquad U = \frac{u + v}{\lambda},\qquad a,b,v\in \mathbb{R},\; \lambda \in \mathbb{R}^+.
\end{equation}

Multiplying Burgers' equation \eqref{Burgers} by a test function $\phi(x) \in C_c^\infty(\mathbb{R})$ and integrating over $\mathbb{R}$, we obtain the weak form 
\begin{equation}\label{Burgers weak form}
\int_{-\infty}^\infty u_t \phi\, \mathrm{d}x = \int_{-\infty}^\infty \bigg(-\nu u_x + \frac{u^2}{2}\bigg)\phi_x\, \mathrm{d}x.
\end{equation}

In the following, we consider the semi-discretization of Burgers' equation where the spatial variable $x$ is discretized and the time variable $t$ remains a continuous variable. In this setting, the interpolating coefficients in the approximation \eqref{ud} of the solution now become functions of $t$:
\begin{equation}\label{approximation}
u(x,t) \approx u^\md(x,t) = \sum_{k=-\infty}^\infty u_k(t) \phi_k(x).
\end{equation}
Substituting \eqref{approximation} into the weak form \eqref{Burgers weak form} and replacing the test function by the hat function $\phi_\ell$, we obtain
\begin{equation}\label{Burgers weak form approximation}
\int_{-\infty}^\infty u_t^\md \phi_\ell\, \mathrm{d}x = \int_{-\infty}^\infty \bigg(-\nu u_x^\md + \frac{(u^\md)^2}{2}\bigg)\phi^\prime_\ell\, \mathrm{d}x,
\end{equation}
where
\[
u_t^\md = \sum_{k=-\infty}^\infty \frac{\md u_k}{\md t} (t)\phi_k(x)\qquad\text{and}\qquad
u_x^\md = \sum_{k=-\infty}^\infty u_k(t) \phi_k^\prime(x).
\]
Under the group action \eqref{Burgers group action}, the differentials $\mathrm{d}x$ and $\mathrm{d}t$ transform according to
\begin{equation}\label{Burgers lifted forms}
\omega^x = \mathrm{D}_x(X)\, \mathrm{d}x + \mathrm{D}_t(X)\, \mathrm{d}t = \lambda(\mathrm{d}x + v\, \mathrm{d}t),\qquad
\omega^t = \mathrm{D}_x(T)\, \mathrm{d}x + \mathrm{D}_t(T)\, \mathrm{d}t = \lambda^2\, \mathrm{d}t,
\end{equation}
where $\mathrm{D}_x$ and $\mathrm{D}_t$ are the total derivative operators in the independent variables $x$ and $t$, respectively, \cite{FO-1999}.  Dual to the one-forms \eqref{Burgers lifted forms} are the implicit derivative operators
\[
\mathrm{D}_X = \frac{1}{\lambda} \mathrm{D}_x,\qquad \mathrm{D}_T = \frac{1}{\lambda^2}(\mathrm{D}_t - v\, \mathrm{D}_x).
\]
Therefore, the hat functions and their first derivatives transform according to 
\[
\Phi_\ell = \phi_\ell \qquad \text{and}\qquad \Phi^\prime_\ell = \mathrm{D}_X(\Phi_\ell) =  \frac{\phi_\ell^\prime}{\lambda}.
\]
Finally, we have
\begin{align*}
U^\md &= \sum_{k=-\infty}^\infty U_k(T) \Phi_k(X) = \sum_{k=-\infty}^\infty \frac{u_k+v}{\lambda}\, \phi_k = \frac{u^\md + v}{\lambda}, \\
U^\md_T &= \sum_{k=-\infty}^\infty \mathrm{D}_T(U_k)\Phi_k = \sum_{k=-\infty}^\infty \frac{1}{\lambda^3}\frac{\md u_k}{\md t} \phi_k = \frac{1}{\lambda^3}\, u_t^\md,\\
U^{\md}_X &= \sum_{k=-\infty}^\infty U_k\, \mathrm{D}_X(\Phi_k) = \sum_{k=-\infty}^\infty \frac{u_k+v}{\lambda^2}\, \phi_k^\prime = \frac{1}{\lambda^2}\, u_x^\md,
\end{align*}
where we used the fact that $\displaystyle \sum_{k=-\infty}^\infty \phi_k = 1$ and $\displaystyle \sum_{k=-\infty}^\infty \phi_k^\prime = 0$, where the sums are defined.  

We now act on the discrete weak form \eqref{Burgers weak form approximation} with the symmetry group \eqref{Burgers group action}.  Since the weak form is evaluated at a fixed time, we substitute
\[
\omega^x = \lambda(\mathrm{d}x + v\, \mathrm{d}t) \equiv \lambda\, \mathrm{d}x
\]
into the transformed weak form.  After simplification, we obtain
\begin{equation}\label{burgers transformed discrete weak form}
\int_{-\infty}^\infty u_t^\md\phi_\ell\, \md x = \int\bigg[-\nu u_x^\md + \frac{(u^\md + v)^2}{2}\bigg]\phi_\ell^\prime \, \md x 
= \int_{-\infty}^\infty \bigg[-\nu u_x^\md + \frac{(u^\md)^2}{2} + v u^\md\bigg]\phi_\ell^\prime\, \md x.
\end{equation}
Due to the occurrence of the Galilean boost parameter $v$, we conclude that the discrete weak form \eqref{Burgers weak form approximation} is not invariant under the symmetry subgroup \eqref{Burgers group action}.

\subsubsection{Symmetry-Preserving Lagrangian Scheme}

In this section we introduce a discrete weak form of Burgers' equation that preserves the symmetry subgroup \eqref{Burgers group action}.  This is done by using the \emph{Lagrangian form} of Burgers' equation given by
\begin{equation}\label{Burgers Lagrangian form}
\frac{\mathrm{d}u}{\mathrm{d}t} = \nu u_{xx},\qquad \frac{\mathrm{d}x}{\mathrm{d}t} = u,
\end{equation}
where
\[
\frac{\mathrm{d}}{\mathrm{d}t} = \mathrm{D}_t + \frac{\mathrm{d}x}{\mathrm{dt}}\mathrm{D}_x = \mathrm{D}_t + u\, \mathrm{D}_x.
\]
In this setting, $x$ is now a function of the time variable $t$.  Therefore, the nodes $x_\ell$ are functions of $t$ and the element $[x_{\ell},x_{\ell+1}]$ varies as a function of time.

Following the general procedure introduced in the previous sections, the first step in constructing a symmetry-preserving weak form consists of computing a discrete moving frame.  Assuming, for simplicity, that
\[
u_x^\ell = \frac{u_{\ell+1}-u_{\ell-1}}{x_{\ell+1}-x_{\ell-1}} > 0\qquad \text{for all}\qquad \ell \in \mathbb{Z}, 
\]
we introduce the cross-section
\[
\mathcal{K} = \{ x_\ell = t = u_\ell = 0,\, u_x^\ell = 1\}.
\]
Solving the normalization equations
\[
0 = \lambda (x_\ell + vt)+a,\qquad 0=\lambda^2t+b,\qquad 0 = \lambda^{-1}(u_\ell + v),\qquad 1=\lambda^{-2} u_x^\ell,
\]
for the group parameters, we obtain the moving frame
\begin{equation}\label{burgers moving frame}
a = -\sqrt{u_x^\ell} (x_\ell - t\, u_\ell),\qquad b = -t\, u_x^\ell,\qquad v = -u_\ell,\qquad \lambda = \sqrt{u_x^\ell}.
\end{equation}

Since $u_k$ and $x_k$ are functions of $t$, we now wish to invariantize $\md u_k/\md t$ and $\md x_k/\md t$.  Under the symmetry group action \eqref{Burgers group action}, we have
\begin{align*}
g\cdot \frac{\md u_k}{\md t} &= \frac{\md U_k}{\md T} = \frac{1}{\lambda^2} \frac{\md}{\md t} \bigg[\frac{u_k+v}{\lambda}\bigg] = \frac{1}{\lambda^3}\frac{\md u_k}{\md t},\\
g\cdot \frac{\md x_k}{\md t} &= \frac{\md X_k}{\md T} = \frac{1}{\lambda^2} \frac{\md}{\md t} [\lambda(x_k+vt)+a] = \frac{1}{\lambda}\bigg(\frac{\md x_k}{\md t} + v\bigg),
\end{align*}
where
\[
\frac{\md}{\md T} = \frac{1}{\lambda^2} \frac{\md }{\md t}
\]
is the derivative operator dual to the one-form $\omega^t = \lambda^2 \mathrm{d}t$.  Using the moving frame \eqref{burgers moving frame}, we have
\[
\iota_\ell\bigg(\cfrac{\mathrm{d}u_k}{\mathrm{d}t}\bigg) = \frac{1}{(u_x^\ell)^{3/2}} \frac{\mathrm{d}u_k}{\mathrm{d}t},\qquad 
\iota_\ell\bigg(\cfrac{\mathrm{d}x_k}{\mathrm{d}t}\bigg) = \frac{1}{\sqrt{u_x^\ell}}\bigg(\frac{\mathrm{d}x_k}{\mathrm{d}t} - u_\ell\bigg).
\]
Next, invariantizing the discrete weak form \eqref{Burgers weak form approximation}, which is obtained by substituting the group normalizations \eqref{burgers moving frame} into the transformed discrete weak form \eqref{burgers transformed discrete weak form}, we obtain the symmetry-preserving discrete weak form
\[
\int_{-\infty}^\infty u_t^\md \phi_\ell\, \md x = \int_{-\infty}^\infty \bigg(-\nu u_x^\md + \frac{(u^\md)^2}{2}-u_\ell\, u^\md\bigg)\phi_\ell^\prime\, \md x.
\]
Evaluating the integrals, and simplifying the expressions, we obtain the symmetry-preserving finite element scheme
\begin{subequations}\label{invariant scheme}
\begin{equation}\label{invariant scheme eq 1}
\frac{1}{3}\bigg[\frac{x_\ell-x_{\ell-1}}{x_{\ell+1}-x_{\ell-1}}\cdot \frac{\mathrm{d}u_{\ell-1}}{\mathrm{d}t} + 2 \cdot \frac{\mathrm{d}u_\ell}{\mathrm{d}t} + \frac{x_{\ell+1}-x_\ell}{x_{\ell+1}-x_{\ell-1}} \cdot \frac{\mathrm{d}u_{\ell+1}}{\mathrm{d}t} \bigg] = \nu u_{xx}^\ell - \bigg(\frac{u_{\ell+1}-2u_\ell+u_{\ell-1}}{3}\bigg)u_x^\ell,
\end{equation}
where
\[
u_{xx}^\ell = \frac{2}{x_{\ell+1}-x_{\ell-1}}\bigg[\bigg(\frac{u_{\ell+1}-u_\ell}{x_{\ell+1}-x_\ell}\bigg) - \bigg(\frac{u_\ell-u_{\ell-1}}{x_\ell-x_{\ell-1}}\bigg) \bigg].
\]
In the Lagrangian formalism, we need to supplement \eqref{invariant scheme eq 1} with a mesh equation that will describe how the node $x_\ell$ will evolve as a function of time.  This can be achieved, in a symmetry-preserving fashion, by setting $\iota_\ell\big(\mathrm{d}x_\ell/\mathrm{d}t\big)=0$, which yields the invariant differential equation
\begin{equation}\label{invariant scheme eq 2}
\frac{\mathrm{d}x_\ell}{\mathrm{d}t} = u_\ell.
\end{equation}
\end{subequations}
In the continuous limit, the invariant scheme \eqref{invariant scheme} converges to \eqref{Burgers Lagrangian form}.

\begin{remark}
In equation \eqref{invariant scheme eq 1} there is no built-in term that would allow to control the evolution of the mesh.  The limit only holds provided the mesh points satisfy equation \eqref{invariant scheme eq 2}.  This is to be expected as we have invariantized the discrete weak form \eqref{Burgers weak form approximation}, defined on a fixed mesh together with the mesh equation $\mathrm{d}x_\ell/\mathrm{d}t = 0$, which forces the nodes to stay fixed as the time variable evolves.  
\end{remark}

Due to the Lagrangian mesh equation \eqref{invariant scheme eq 2}, the invariant scheme \eqref{invariant scheme} will in general suffer from mesh tangling and singularities~\cite{HR-2011}.  To avoid these problems, we now construct a symmetry-preserving finite element scheme that will hold on any moving mesh.

\subsubsection{Symmetry-Preserving $r$-Adaptive Scheme}

In this section we construct a symmetry-preserving finite element scheme with a built-in term that takes into account the evolution of the mesh.  This is achieved by invariantizing
\begin{equation}\label{burgers moving mesh weak form}
\int_{-\infty}^\infty \big(u_t^\md - u^\md_x x_t\big)\phi_\ell\, \mathrm{d}x = 
\int_{-\infty}^\infty \bigg(-\nu u_x^\md + \frac{(u^\md)^2}{2}\bigg)\phi_\ell^\prime\, \mathrm{d}x,
\end{equation}
where the extra term on the left-hand side of equation \eqref{burgers moving mesh weak form} takes into account the movement of the mesh, and where
$$
x_t = \sum_{k=-\infty}^\infty \frac{\mathrm{d}x_k}{\mathrm{d}t}\phi_k.
$$
In particular, we note that when $\md x_k/\md t = 0$ for all $k$, then we recover the discrete weak form \eqref{Burgers weak form approximation} of Burgers's equation on a fixed mesh.  

Under the group action \eqref{Burgers group action},
\[
X_T = \sum_{k=-\infty}^\infty \frac{\md X_k}{\md T} \Phi_k = \sum_{k=-\infty}^\infty \frac{1}{\lambda}\bigg(\frac{\md x_k}{\md t} + v\bigg)\phi_k = \frac{x_t+v}{\lambda}.
\]
Therefore, acting by the symmetry group \eqref{Burgers group action} on the discrete weak form \eqref{burgers moving mesh weak form} we obtain the transformed weak form
\begin{equation}\label{transformed burgers moving mesh weak form}
\int_{-\infty}^\infty \big[u_t^\md - u_x^\md(x_t+v)\big]\phi_\ell\, \mathrm{d}x = \int_{-\infty}^\infty\bigg[-\nu u_x^\md + \frac{(u^\md)^2}{2} + v u^\md\bigg]\phi_\ell^\prime\, \mathrm{d}x.
\end{equation}
The invariantization of \eqref{burgers moving mesh weak form} is obtained by substituting the moving frame expressions \eqref{burgers moving frame} into \eqref{transformed burgers moving mesh weak form}, which yields the symmetry-preserving discrete weak form
\[
\int_{-\infty}^\infty \big[u_t^\md - u_x^\md(x_t-u_\ell)\big]\phi_\ell\, \mathrm{d}x = \int_{-\infty}^\infty\bigg[-\nu u_x^\md + \frac{(u^\md)^2}{2} - u_\ell\, u^\md\bigg]\phi_\ell^\prime\, \mathrm{d}x.
\]
%
Evaluating the integrals, we obtain the symmetry-preserving finite element scheme
\begin{multline*}
\frac{x_\ell-x_{\ell-1}}{3(x_{\ell+1}-x_{\ell-1})} \cdot \frac{\mathrm{d}u_{\ell-1}}{\mathrm{d}t}
+ \frac{2}{3} \cdot \frac{\mathrm{d}u_\ell}{\mathrm{d}t}
+ \frac{x_{\ell+1}-x_\ell}{3(x_{\ell+1}-x_{\ell-1})} \cdot \frac{\mathrm{d}u_{\ell+1}}{\mathrm{d}t}
-\frac{u_\ell-u_{\ell-1}}{3(x_{\ell+1}-x_{\ell-1})}\cdot \frac{\mathrm{d}x_{\ell-1}}{\mathrm{d}t} \\
- \frac{2}{3}u_x^\ell \cdot \frac{\mathrm{d}x_\ell}{\mathrm{d}t} 
- \frac{u_{\ell+1}-u_\ell}{3(x_{\ell+1}-x_{\ell-1})}\cdot \frac{\mathrm{d}x_{\ell+1}}{\mathrm{d}t} = \nu\, u_{xx}^\ell - \bigg(\frac{u_{\ell+1} + u_\ell+u_{\ell-1}}{3}\bigg)u_x^\ell.
\end{multline*}
The remaining ingredient is to prescribe $\mathrm{d}x_{\ell}/\mathrm{d}t$ using an invariant mesh equation.  For example, when using the mesh equation \eqref{invariant scheme eq 2}, the scheme reduces to 
\[
\frac{x_\ell-x_{\ell-1}}{3(x_{\ell+1}-x_{\ell-1})} \cdot \frac{\mathrm{d}u_{\ell-1}}{\mathrm{d}t}
+ \frac{2}{3} \cdot \frac{\mathrm{d}u_\ell}{\mathrm{d}t}
+ \frac{x_{\ell+1}-x_\ell}{3(x_{\ell+1}-x_{\ell-1})} \cdot \frac{\mathrm{d}u_{\ell+1}}{\mathrm{d}t} = \nu\, u_{xx}^\ell.
\]
Other invariant mesh equations for Burgers' equation were proposed e.g.\ in~\cite{BN-2014,BV-2017}, and will not be discussed further here.


\section{Conclusions and Outlook}\label{section conclusion}

In this paper we have, for the first time, laid out a partial theory for constructing symmetry-preserving finite element schemes. This contribution is timely given the large body of literature that exists nowadays regarding the construction of symmetry-preserving finite difference schemes, and due to the obvious importance that finite element discretizations play in mathematical sciences.
 
While we have primarily restricted our attention to second-order differential equations, the principles introduced in this paper are applicable to higher-order differential equations, boundary value problems (though this usually reduces the size of the admitted symmetry group) as well as to multi-dimensional systems of partial differential equations.  A main complication when tackling higher-order differential equations is the necessity to use higher-order basis functions. Conceptually, these higher-order basis functions can readily be included in the theory laid out in the present paper.  Since the resulting computations substantially grow in complexity, we have however abstained from including them here for the sake of clarity of this first exposition on invariant finite element methods.
 
Invariant discretization schemes are a particular class of geometric numerical integrators that are designed to preserve at the discrete level (a subgroup of) the maximal Lie symmetry group of a system of differential equations. The motivation for the development of geometric numerical integrators is that, in general, maintaining the intrinsic geometric properties of a system of differential equations improves the long-term behavior of a numerical integration scheme. In the case of symmetries, it has been shown that invariant integrators play an essential role for blow-up problems, where they have been shown to outperform standard non-invariant integrators. The preservation of symmetries in finite element schemes opens up the possibility to compare invariant finite element schemes against non-invariant finite element discretizations, which has been done for a single example in the present work.  We reserve a more detailed comparison for future work.

\subsection*{Acknowledgements}

This research was undertaken, in part, thanks to funding from the Canada Research Chairs program, the NSERC Discovery Grant program and the LeverageR{\&}D program of the Research and Development Corporation of Newfoundland and Labrador.  FV would like to thank Memorial University of Newfoundland, where this research was initiated, for the hospitality during his stay.
 

\end{document}